\documentclass[reqno]{amsart}
\usepackage{amssymb,mathrsfs,mathtools,etoolbox,enumitem}
\usepackage{hyperref}
\hypersetup{
	colorlinks	= true,  
	urlcolor     = blue, 
	linkcolor   = blue,  
	citecolor   = red    
}

\newtheorem{thm}{Theorem}[section]
\newtheorem{prop}[thm]{Proposition}
\newtheorem{lemma}[thm]{Lemma}

\theoremstyle{definition} 
\newtheorem{remark}[thm]{Remark}

\numberwithin{equation}{section}
\counterwithout{footnote}{section}

\newcommand{\Fourier}{\mathcal{F}}
\newcommand{\fourier}[1]{\widehat{#1}}
\newcommand{\dm}{\,\mathrm{d}}
\newcommand{\dx}{\dm x}

\newcommand{\T}{\mathbb{T}}
\newcommand{\R}{\mathbb{R}}
\newcommand{\N}{\mathbb{N}}

\DeclareMathOperator{\supp}{supp}

\newcommand\norm[2][]{
	\ifstrempty{#1}{\left\| {#2} \right\|_{L^2} }{ \left\| {#2} \right\|_{L^{#1}} } }

\allowdisplaybreaks


\title[Maximum height wave]{Waves of maximal height for a class of nonlocal equations with inhomogeneous symbols}

\author{Hung Le}
\address{Department of Mathematical Sciences, Norwegian University of Science and Technology, 7491 Trondheim, Norway}
\email{hung.le@ntnu.no}
\thanks{HL was supported by the ERCIM `Alain Bensoussan' Fellowship Programme.}

\date{\today}

\begin{document}
	
\begin{abstract}
	In this paper, we consider a class of nonlocal equations where the convolution kernel is given by a Bessel potential symbol of order $\alpha$ for $\alpha > 1$.  Based on the properties of the convolution operator, we apply a global bifurcation technique to show the existence of a highest, even, $2\pi$-periodic traveling-wave solution.  The regularity of this wave is proved to be exactly Lipschitz.  
\end{abstract}

\maketitle

\section{Introduction}

Our object of study is the following nonlocal equation
\begin{equation} \label{main ut eqn}
u_t + L u_x + u u_x = 0, \qquad (t,x) \in \R \times \R,
\end{equation}
where $L$ denotes the Fourier multiplier operator with symbol $m(\xi) = (1 + \xi^2)^{-\alpha/2}$, $\alpha > 1$.  Equation~\eqref{main ut eqn} is known as the \emph{fractional Korteweg--de Vries equation}.  We are looking for a $2\pi$-periodic traveling-wave solution $u(t,x) = \phi(x - \mu t)$, where $\mu > 0$ denotes the speed of the right-propagating wave.  In this moving frame, equation \eqref{main ut eqn} takes the form
\begin{equation} \label{main u eqn}
-\mu \phi + L \phi + \frac{1}{2} \phi^2 = 0,
\end{equation}
where the equation has been integrated once, and the constant of integration set to zero.  In fact, there is no loss of generality in doing so due to the Galilean change of variables:
\[
\phi \mapsto \phi + \gamma, \quad \mu \mapsto \mu + \gamma, \quad B \mapsto B + \gamma \left( 1 - \mu - \frac{1}{2} \gamma \right)
\]
for any $\gamma \in \R$, which maps solutions of $-\mu \phi + L \phi + \frac{1}{2} \phi^2 = B$ to solutions of a new equation of the same form.  In this paper, the above multiplier $m(\xi)$ has order of $-\alpha$, which is less than $-1$.  This type of dispersive terms tends to behave smoothly at far field and worse at the origin.  Solutions to the type of equation \eqref{main ut eqn} with negative orders feature wave breaking and peaked periodic waves; see discussions in \cite{Ehrnstrom_Wang2020,Geyer_Pelinovsky2019}.  It should be mentioned that for a linear operator of order strictly smaller than $-1$, the regularity of the highest, periodic, traveling-wave solution does not depend on the order operator; see Theorem \ref{Thm:cannot-C1}.  Instead, the dispersive term may affect the angle at the wave crest.

Investigation on highest waves for various types of water-waves equations has a long history.  One of the earliest works to study peaked waves was Stokes, who \cite{Stokes1880} conjectured that for a fixed wavelength and the gravitational acceleration, the Euler equation has a highest, periodic traveling-wave with a sharp crest of angle $2\pi/3$.  However, attempts to prove the Stokes conjecture had failed until the work by Amick, Fraekel, and Toland~\cite{Amick_Fraenkel_Toland1982}.  Whitham~\cite{Whitham1974} also made an assumption that the solution $\phi$ to the Whitham equation reached its critical height when $\phi = \mu$ (the notation has been changed to match equation~\eqref{main u eqn}).  Recently, Ehrnstr\"om and Wahl\'en~\cite{Ehrnstrom_Wahlen2019} confirmed the existence and regularity of a highest and periodic traveling-wave solution for the Whitham equation, and hence successfully proved Whitham conjecture.  This solution features a sharp crest and is exactly $C^{1/2}$-H\"older continuous at the peak and smooth between its crest and trough on the half-period.  After this breakthrough work, there have been an increasing number of studies on nonlocal equations using their approach.  One of the earliest papers in this direction is to justify the convexity of the wave profile between consecutive crests using a computer-assisted proof \cite{Encisco_Gomez_Vergara2018}.  A similar result as in \cite{Ehrnstrom_Wahlen2019} but for solitary waves was also obtained by Truong, Wahl\'en, and Wheeler \cite{Truong_Wahlen_Wheeler2020} using a center manifold theorem approach.  Moreover, Ehrnstr\"om, Johnson, and Claasen~\cite{Ehrnstrom_Johnson_Claassen2019} proved the existence and regularity of a highest wave for the bidirectional Whitham equation with cubic nonlinearity and a Fourier multiplier symbol
$m(\xi) = \tanh(\xi) / \xi$.  They showed that the highest cusped wave has a singularity at its crest of the form $|x \log(|x|)|$.  Utilizing the same approach as in \cite{Ehrnstrom_Wahlen2019}, Arnesen~\cite{Arnesen2019} considered the Degasperis--Procesi equation.  This equation has a nonlocal form with a quadratic nonlinearity and a Fourier multiplier symbol $m(\xi) = (1 + \xi^2)^{-1}$.  Using the method in \cite{Ehrnstrom_Wahlen2019}, Arnesen constructed a family of highest, peaked, periodic traveling waves and showed that any highest, even, periodic wave of the Degasperis--Procesi equation is exactly Lipschitz at its crest.  The author, therefore, excluded the existence of even, periodic, cusped traveling-wave solutions.  

Since the unidirectional Whitham equation is nonlocal, the equation~\eqref{main ut eqn} is obtained, up to a scaling factor, by replacing the Whitham dispersion relation by a more regular dispersive term.  Building on the integral kernel $K$, we prove the existence and regularity properties of a highest, $2\pi$-periodic traveling-wave solution of equation \eqref{main u eqn}.  Our most direct influence are the works by Ehrnstr\"om and Wahl\'en~\cite{Ehrnstrom_Wahlen2019} and by Bruell and Dhara~\cite{Bruell_Dhara2018}.  In the former work, the investigation started with the Whitham equation $$u_t + 2 uu_x + Lu_x = 0,$$ where $L$ is the Fourier multiplier operator given by the symbol $$m(\xi) = \sqrt{\frac{\tanh\xi}{\xi}}.$$  This multiplier has order greater than $-1$, is inhomogeneous and completely monotone.  The main result is the existence of a highest, cusped, $P$-periodic traveling-wave solution, which is even, strictly decreasing, smooth on each half-period, and belongs to the H\"older space $C^\frac{1}{2}(\R)$.  The proof is based on the regularity and monotonicity properties of the convolution kernel induced by $m$.  On the other hand, Bruell and Dhara~\cite{Bruell_Dhara2018} considered the fractional KdV equation as in equation~\eqref{main ut eqn}, where $L$ is the Fourier multiplier operator given by the homogeneous symbol 
\[
m(\xi) = |\xi|^{-\alpha}, \qquad \alpha > 1.
\]
This multiplier has order less than $-1$, is homogeneous, and completely monotone.  The authors showed the existence of a highest, $2\pi$-periodic traveling waves solution, which is Lipschitz continuous.  Their method is based on the nonlocal approach in the paper~\cite{Ehrnstrom_Wahlen2019} but adapt to treat homogeneous symbols.  The convolution kernel associated with this homogeneous symbol cannot be identified with a positive, decaying function on the real line.  To compensate for this lack of positivity, solutions are assumed to have the zero mean, and hence the constant of integration $B$ in equation~\eqref{main u eqn} now becomes $\fourier{\phi}(0)$.  

In the present paper, we will combine the approach in \cite{Ehrnstrom_Wahlen2019} to treat inhomogeneous symbols, and the derivation for an estimate of a priori traveling-wave solution to nonlocal equations as in \cite{Bruell_Dhara2018}.  We recall that the symbol $m$ for the fractional KdV equation~\eqref{main ut eqn} is Bessel potential function, which is inhomogeneous and of order strictly smaller than $-1$.  We also do not make the assumption of zero mean solution as in \cite{Bruell_Dhara2018}.  Since the kernel operator is positive, we choose the constant of integration in equation~\eqref{main u eqn} to be $B = 0$.  Moreover, even though our kernel $K$ is even and smooth for all $x \in \R$ with rapid decay derivatives, we do not use a closed formula for $K$ in the physical space, as well as for its periodization as in \cite{Bruell_Dhara2018,Ehrnstrom_Wahlen2019}.  It is also worth to mention that the kernel $K$ does not exhibit the completely monotone feature, which plays a main role for the analysis in \cite{Bruell_Dhara2018,Ehrnstrom_Wahlen2019}.  Instead, our investigation only relies on its qualitative behaviors such as positivity of $K$ and the signs of its derivatives; see Section~\ref{Section:Setting:K*acts:period} and \ref{Section:K-properties}.  The properties of the convolution kernel are proved in the ongoing work by Bruell, Ehrnstr\"om, Johnson, and Wahl\'en \cite{Bruell_Ehrnstrom_Johnson_Wahlen}.  

We now briefly discuss the strategies we used to derive the results.  We will be using a global bifurcation theory as in the paper \cite{Buffoni_Toland2003} to build a global, locally analytic curve of periodic smooth waves where the solution maximum approaches $\mu$.  We further show that the only possibility along the main bifurcation curve is that $\max \phi$ must approach $\mu$.  Moreover, we investigate the solution $\phi$ satisfying $\max \phi = \mu$.  Using the properties of the convolution operator $K$, we obtain Lipschitz continuous at its crest.

Let us formulate our main existence theorem which constructs solutions bifurcating from nontrivial, even, smooth, and $2\pi$-periodic traveling-wave solutions of equation~\eqref{main ut eqn}.  At the end of the bifurcation curve, the limiting solution reaches its highest and is precisely Lipschitz continuous.

\begin{thm}[Main theorem] \label{Thm:main}
	For each integer $k \ge 1$, there exists a wave speed $\mu^*_k > 0$ and a global bifurcation branch:
	\[
	s \mapsto (\phi_k(s), \mu_k(s)), \qquad s > 0,
	\]
	of nontrivial, $\frac{2\pi}{k}$-periodic, smooth, even solutions to the steady equation \eqref{main u eqn} for $\alpha > 1$, emerging from the bifurcation point $(0, \mu_k^*)$.  Moreover, given any unbounded sequence $(s_n)_{n \in \N}$ of positive numbers $s_n$, there exists a subsequence of $(\phi_k(s_n))_{n \in \N}$, which converges uniformly to a limiting traveling-wave solution $(\bar \phi_k, \bar \mu_k)$ that solves \eqref{main u eqn} and satisfies
	\[
	\bar \phi_k(0) = \bar \mu_k.
	\]
	The limiting wave is strictly increasing on $\left(-\frac{\pi}{k}, 0 \right)$ and exactly Lipschitz at $x \in \frac{2\pi}{k} \mathbb{Z}$.
\end{thm}

The outline of our study is as follows.  In Section~\ref{Section:Setting:K*acts:period} we introduce notations and functional-analytic settings that we encounter in the analysis.  We also inspect the convolution kernel $K$ corresponding to the symbol $m(\xi)$ and show that the action of the operator $K$ on periodic functions on the real line is in fact similar to the \emph{periodized kernel} $K_P$ acts on these functions in one period.  In Section~\ref{Section:K-properties}, we discuss some interesting analytic features of the periodized kernel $K_P$.  In fact, both $K$ and $K_P$ share many properties.

Section~\ref{Section:Apriori-properties:soln} is the heart of this paper.  We provide some general estimates of an even solution to equation~\eqref{main u eqn}.  In particular, using the regularity and monotonicity properties of the convolution kernel $K_P$, we prove that if a traveling-wave solution, which is even, periodic, and monotone on half the period, attains its maximum equal to the wave speed $\mu$, then it is Lipschitz continuous.  Finally, Section~\ref{Section:Global:bifurcation} proved our main theorem, Theorem~\ref{Thm:main}.  It consists of the bifurcation analysis from  \cite{Bruell_Dhara2018,Ehrnstrom_Kalisch2013,Ehrnstrom_Kalisch2009}, in which we rule out certain possibilities for the bifurcation curve.  We, therefore, construct a global curve of sinusoidal, periodic smooth waves along which $\max \phi \to \mu$, and an analysis of solutions at their maximum $\max\phi = \mu$.  

\section{Functional-analytic setting and general conventions}
\label{Section:Setting:K*acts:period}

Setting $\T := [-\pi, \pi]$, we denote by $\mathcal{D}(\T)$ the space of test functions on $\T$, whose dual space, the space of distributions on $\T$, is $\mathcal{D}'(\T)$.  If $\mathcal{S}(\mathbb{Z})$ is the space of rapidly decaying functions from $\mathbb{Z}$ to $\mathbb{C}$ and $\mathcal{S}'(\mathbb{Z})$ denotes its dual space, let $\Fourier : \mathcal{D}'(\T) \to \mathcal{S}'(\mathbb{Z})$ be the Fourier transformation on the torus defined by the duality on $\mathcal{D}(\T)$ via
\[
\Fourier f(\xi) = \fourier f(\xi) := \int_{\mathbb{T}}f(x) \exp(-ix\xi) \dx, \qquad f \in \mathcal{D}(\T).
\]
To simplify our notations, throughout this paper, if $f$ and $g$ are elements in an ordered Banach space, we write $f \lesssim g$ $(f \gtrsim g)$ if there exists a constant $c > 0$ such that $f \le cg$ $(f \ge cg)$.  Moreover, the notation $f \eqsim g$ is used whenever $f \lesssim g$ and $f \gtrsim g$.  We also denote $\R_+ := [0, \infty]$ and by $\N_0$ the set of natural numbers including zero.  The space $\mathcal{L}(X; Y)$ denotes the set of all bounded linear operators from $X$ to $Y$.  

Next, let us investigate how the convolution $K *$ acts on periodic functions.  Suppose that $f \in L^\infty(\R)$ is periodic and even.  Since $K$ is in $L^1(\mathbb{R})$, we can write:
\begin{align*}
\int_{-\infty}^\infty K(x-y) &f(y) \dm y = \sum_{n=-\infty}^{\infty} \int_{-\pi}^\pi K(x - y + 2n\pi) f(y) \dm y \\
&= \int_{-\pi}^\pi \left( \sum_{n=-\infty}^\infty K(x - y + 2n\pi) \right) f(y) \dm y
	=: \int_{-\pi}^\pi K_P(x - y) f(y) \dm y.
\end{align*}
It is clear from the definition of $K_P(x)$ that it is $2\pi$-periodic, even, and continuous on $\T$.  Moreover, Minkowski's inequality shows that $K_P(x)$ belongs to $L^p(-\pi, \pi)$, for $1 \le p < 2$.  Therefore, Carleson--Hunt Theorem~\cite{Carleson_Hunt1982} implies that $K_P(x)$ can be approximated pointwise by its Fourier series:
\[
K_P(x) = \frac{1}{2\pi} \sum_{n \in \mathbb{Z}} \fourier A_n \exp(inx ), \quad a.e.,
\]
and the Fourier coefficients of $K_P$ are given by
\begin{align} \begin{split}
\fourier A_n &= \int_{-\pi}^\pi \sum_{j=-\infty}^\infty K(x + 2j\pi) \exp(-ixn) \dm x \\
&= \sum_{j=-\infty}^\infty \int_{-\pi}^\pi K(x + 2j\pi) \exp(-i(x + 2j\pi)n) \dx \\
&= \int_{-\infty}^\infty K(x) \exp(-ixn) \dx = \fourier{K}(n).
\end{split} \end{align}
Thus, the periodic problem is given by the same multiplier as the problem on the line, so we have the representation:
\[
K * f(x) = \frac{1}{2\pi} \sum_{n \in \mathbb{Z}} \fourier{f}(n) \fourier{K}(n) \exp(inx)
= \frac{1}{2\pi} \sum_{n \in \mathbb{Z}} \fourier{f}(n) \, m(n) \exp(inx).
\]
The \emph{periodized operator} $K_P$ is introduced to aid the analysis of periodic solutions satisfying certain sign conditions in a half-period.  From the above analysis, $K_P$ can be expressed as the Fourier series
\begin{equation} \label{K:Fourier:representation}
K_P(x) = \frac{1}{2\pi} \sum_{n \in \mathbb{Z}} m(n) \exp(inx). 
\end{equation}
and the relationship between $K$ and $K_P$ is
\begin{equation} \label{K:Kp:relationship}
K_P(x) = \sum_{n \in \mathbb{Z}} K(x + 2n\pi).
\end{equation}
Note that this sum is absolutely convergent since $K$ has rapid decay.

\subsection*{The operator $L$}

Now let $L$ be the operator
\[
L: f \mapsto K * f,
\]
defined via duality on the space $\mathcal{S}'(\R)$ of tempered distributions.  Then one can see that for a continuous periodic function $f$, the operator $L$ is given by 
\[
\int_{-\pi}^\pi K_P(x - y) f(y) \dm y,
\]
and more generally by 
\[
\int_\R K(x - y) f(y) \dm y
\]
if $f$ is bounded and continuous.  Thus, we have the following formula for $L$
\begin{equation} \label{L:formula}
Lf(x) = \frac{1}{2\pi} \sum_{n \in \mathbb{Z}} m(n) \fourier f(n) \exp(inx). 
\end{equation}

Next, we shall say that a function $\varphi: \R \to \R$ is \emph{H\"older continuous of regularity} $\beta \in (0,1)$ \emph{at a point} $x \in \R$ if
\[
|\varphi|_{C^\beta_x} := \sup_{h \ne 0} \frac{|\varphi(x + h) - \varphi(x)}{|h|^\beta} < \infty,
\]
and let
\begin{align*}
C^\beta(\R) &= \{ \varphi \in C(\R): \sup_x |\varphi|_{C^\beta_x} < \infty \}
\end{align*}
be the space of $\beta$-H\"older continuous functions on $\R$.  If $k \in \N$, then
\[
C^{k, \beta}(\R) = \{ \varphi \in C^k(\R): \varphi^{(k)} \in C^\beta(\R) \}
\]
denotes the space of $k$-times continuously differentiable functions whose $k$-th derivative is $\beta$-H\"older continuous on $\R$.

Before we proceed to discuss the space of the operator $L$, let us recall the definition of Besov spaces $B^s_{p,q}(\R)$ using the Littlewood--Paley decomposition.  Let $\{\gamma_j\}_{j \in \N}$ be a family of smooth and compactly supported functions, where 
\begin{gather*}
\supp \gamma_0 \subset [-2, 2], \qquad \supp \gamma_j \subset \{ \xi \in \R : 2^{j-1} \le |\xi| \le 2^{j+1} \} \qquad \text{for} \quad j \ge 1, \\
\sum_{j=0}^\infty \gamma_j(\xi) = 1
\end{gather*}
for all $\xi \in \R$.  
For a tempered distribution $f \in \mathcal{S}'(\R)$, we let $\gamma_j(D) f = \Fourier^{-1}(\gamma_j(\xi) \fourier{f}(\xi))$, so that
\[
f = \sum_{j=0}^\infty \gamma_j(D) f.
\]
Then the Besov spaces $B^s_{p,q}(\R)$, $s \in \R$, $1 \le p \le \infty$, $1 \le q < \infty$ are defined by
\[
\left\{ f \in \mathcal{S}'(\R): \|f\|_{B^s_{p,q}(\R)} := \left[ \sum_{j=0}^\infty (2^{sj} \|\gamma_j(D) f \|_{L^p(\R)})^q \right]^{1/2} < \infty \right\},
\]
and for $1 \le p \le \infty$ and $q = \infty$, we instead define
\[
B^s_{p,q}(\R) = \left\{ f \in \mathcal{S}'(\R): \|f\|_{B^s_{p,\infty}(\R)} := \sup_{j \ge 0} 2^{sj} \|\gamma_j(D) f\|_{L^p(\R)} < \infty \right\}.
\]
For a $2\pi$-periodic tempered distribution $f = \frac{1}{2\pi} \sum_{k \in \mathbb{Z}} \fourier{f}_k \exp(ik\cdot)$, we have the identity
\[
\gamma_j(D) f = \frac{1}{\pi} \sum_{k \in \mathbb{Z}} \gamma_j (k) \fourier{f}_k \exp(ikx),
\]
so that $\gamma_j(D) f$ is a trigonometric polynomial.  

Moreover, for $s \in \R$, we define the Zygmund spaces $\mathcal{C}^s$ of order $s$ by
\[
\mathcal{C}^s(\T) = B^s_{\infty,\infty}(\T),
\]
and recall that when $s$ is a positive noninteger, we have $\mathcal{C}^s = C^{\lfloor s \rfloor, s - \lfloor s \rfloor}$;  and when $s$ is a nonnegative integer, we have $W^{s, \infty} \subsetneq \mathcal{C}^s$.  
As a consequence of Littlewood--Paley theory, we have the relation $\mathcal{C}^s(\T) = C^s(\T)$ for any positive noninteger $s$.  Thus, the H\"older spaces on the torus are completely characterized by Fourier series.  It $s \in \mathbb{N}$, then $C^s(\T)$ is a proper subset of $\mathcal{C}^s(\T)$ and
\[
C^1(\T) \subsetneq C^{1-}(\T) \subsetneq \mathcal{C}^1(\T),
\]
where $C^{1-}(\T)$ denotes the space of Lipschitz continuous functions on $\T$.

To finish this section, we discuss the space of the operator $L$.  It follows from the well-known Bessel potential estimate
\[
|D_\xi^n m(\xi)| \lesssim (1 + |\xi|)^{-\alpha - n}, \quad n \ge 0,
\]
that $L$ defines a bounded operator
\[
L: B^s_{p,q}(\T) \to B^{s + \alpha}_{p,q}(\T), 
\]
see, for example, \cite{Arendt_Bu2004,Bahouri_Chemin_Danchin2011}.  In particular, the operators
\[
L: \mathcal{C}^s(\T) \to \mathcal{C}^{s + \alpha}(\T)
\]
is bounded on $\T$.  Thus, $L$ is a smoothing operator of order $-\alpha < -1$.

\section{Properties of the periodized kernel $K_P$} \label{Section:K-properties}

This section is devoted to summarize some properties of the kernel $K_P$.  We recall the Fourier representation of $K_P$ \eqref{K:Fourier:representation} and the relationship \eqref{K:Kp:relationship} between $K$ and $K_P$.  Recalling that $K = \Fourier^{-1}(m)$ and the Fourier multiplier is given by the Bessel potential $m(\xi) = (1 + \xi^2)^{-\alpha/2}$ for $\alpha > 1$, we record some properties of the periodized operator $K_P$.
\begin{thm} \label{Thm:K-Properties}
	Let $\alpha > 1$.  The kernel $K_P$ has the following properties:
	\begin{enumerate}[label=(\alph*)]
		\item $K_P$ is even, positive, $2\pi$-periodic, and continuous.  $K_P$ is smooth on $\T \backslash \{0\}$.
		\item $K_P$ is in $W^{1,1}(\T)$.  In particular, $K_P'$ is integrable.  Moreover, $K_P$ is H\"older continuous with $\beta \in (0, \alpha-1)$ if $\alpha \in (1, 2)$, Lipschitz continuous when $\alpha=2$, and continuously differentiable if $\alpha > 2$.
		\item $K_P$ is decreasing on $(0, \pi)$.
	\end{enumerate}
\end{thm}
\begin{proof}
	We first observe that this theorem holds when we replace $K_P$ by $K$; see, for example, by Grafakos~\cite{Grafakos2014}.  Due to the Fourier representation~\eqref{K:Fourier:representation} and the relation~\eqref{K:Kp:relationship}, the evenness, positivity, periodicity, and regularity properties of $K_P$ are inherited from $K$.  By \cite[Proposition 1.2.5]{Grafakos2014}, we know that $K$ is smooth on $\T \backslash \{0\}$, and so is $K_P$ by the relation~\eqref{K:Kp:relationship}.  Moreover, the regularity of $K_P$ in part (b) follows from the regularity of $K$ in \cite[Section 1.3]{Grafakos2014}.  Part (c), however, is proved in the ongoing work by Bruell, Ehrnstr\"om, Johnson, and Wahl\'en~\cite{Bruell_Ehrnstrom_Johnson_Wahlen}.
\end{proof}

We know from the above theorem that $K_P$ is even and decreasing on $(0, \pi)$.  The next lemma provides monotonicity property for the operator $L = K*$ based on $K_P$.
\begin{lemma} \label{Lemma:L:monotone}
	Let $\alpha > 1$.  The operator $L$ is parity preserving on $L^\infty(\T)$.  Moreover, if $f, g \in L^\infty(\T)$ are odd functions satisfying $f(x) \ge g(x)$ on $[0,\pi]$, then either
	\[
	Lf (x) > Lg(x) \qquad \text{for all } \quad x \in (0,\pi),
	\]
	or $f = g$ on $\T$.  
\end{lemma}

\begin{proof}
	To see that $L$ is parity-preserving, let $f \in L^\infty(\T)$ be an odd function.  Then since $K$ is even, we have
	\begin{align*}
	Lf(x) + Lf(-x) &= \int_{-\pi}^\pi K_P(x - y) f(y) \dm y + \int_{-\pi}^\pi K_P(-x - y) f(y) \dm y \\
	&= \int_{-\pi}^\pi K_P(x - y) (f(y) + f(-y)) \dm y = 0,
	\end{align*}
	which shows that $Lf$ is odd.  Similarly, when $f$ is even, $Lf$ is also even.  
	
	To show the second part of the lemma, let $f, g \in L^\infty(\T)$ be odd functions such that $f(x) \ge g(x)$ on $[0,\pi]$.  By contradiction, assume that there exists $x_0 \in (0, \pi)$ such that $Lf(x_0) = Lg(x_0)$.  Then we compute
	\begin{equation*} \begin{split} \label{int:K(x-y),K(x+y):L:preserving}
	Lf(x_0) - Lg(x_0) &= \int_{-\pi}^\pi K_P(x_0 - y) (f(y) - g(y)) \dm y \\
	&= \int_0^\pi (K_P(x_0 - y) - K_P(x_0 + y)) (f(y) - g(y)) \dm y.
	\end{split} \end{equation*}
	Since $0 < x_0, y < \pi$, we have
	\[
	-\pi < x_0 - y < x_0 + y < 2\pi, \quad -\pi < x_0 - y < \pi, \quad 0 < x_0 + y < 2\pi.
	\]
	Then the evenness, periodicity, and monotonicity of $K_P$ yield
	\[
	K_P(x_0 - y) - K_P(x_0 + y) > 0
	\]
	for all $x_0, y$ in the interval $(0,\pi)$, and therefore, $Lf(x_0) > Lg(x_0)$, which is a contradiction to our assumption $Lf(x_0) = Lg(x_0)$, unless $f = g$ on $\T$.  
\end{proof}

\section{A priori properties of periodic traveling-wave solutions} \label{Section:Apriori-properties:soln}

In this section, we provide some basic properties of solutions to the main equation~\eqref{main u eqn} such as its a priori estimate, monotonicity, and regularity.  When we say $\phi$ is a solution, $\phi$ must be real-valued, bounded, and satisfy equation~\eqref{main u eqn} pointwise.  Fixing $\alpha > 1$, to aid our analysis, we re-write equation~\eqref{main u eqn} as
\begin{equation} \label{eqv u eqn}
\frac{1}{2} (\mu - \phi)^2 = \frac{1}{2} \mu^2 - L \phi.
\end{equation}
Our first result is a rough bound for a solution according to the wave speed $\mu$.
\begin{lemma} \label{Lem:maxPhi:minPhi:estimate}
	Let $\phi \in C(\T)$ be solution of \eqref{eqv u eqn}.  Then if $\mu > 1$, then
	\begin{equation*} \label{tmp:ineq-xm}
	0 \le \phi(x_m) \le 2(\mu - 1) \le \phi(x_M)
	\end{equation*}
	and if $\mu \le 1$, then
	\begin{equation*} \label{tmp:ineq-xM}
	2(\mu - 1) \le \phi(x_m) \le 0 \le \phi(x_M),
	\end{equation*}
	where $\phi(x_m) := \min_{x \in \T} \phi(x)$ and $\phi(x_M) := \max_{x \in \T} \phi(x)$.  
\end{lemma}
\begin{proof}
	By Lemma~\ref{Lemma:L:monotone}, $L$ is a strictly monotone operator.  Since furthermore $Lc = c$ for constants $c$, we therefore obtain that
	\[
	\frac{1}{2} \left( \mu - \phi \right)^2 = \frac{1}{2} \mu^2 - L \phi \le \frac{1}{2} \mu^2 - \phi(x_m).
	\]
	In particular, $\frac{1}{2} (\mu - \phi(x_m))^2 \le \frac{1}{2} \mu^2 - \phi(x_m)$, or equivalently, we have the inequality
	\[
	\phi(x_m)\left( \frac{1}{2} \phi(x_m) - (\mu - 1) \right) \le 0.
	\]
	Similar arguments can be made to obtain 
	\[
	\frac{1}{2} \left( \mu - \phi \right)^2 = \frac{1}{2} \mu^2 - L \phi \ge \frac{1}{2} \mu^2 - \phi(x_M),
	\]
	which yields that
	\[
	\phi(x_M)\left( \frac{1}{2} \phi(x_M) - (\mu - 1) \right) \ge 0.
	\]
	Combining both cases, we have the desired estimates.  
\end{proof}
\begin{remark} \label{Rmk:phi=0}
	If a solution satisfies $\phi(x)=0$ for some $x$, then at those $x$, the equation~\eqref{eqv u eqn} reduces to $L\phi=0$.  Since $L$ is a strictly monotone operator by Lemma~\ref{Lemma:L:monotone}, we must have either $\phi \equiv 0$, or $\phi$ is sign-changing.  
\end{remark}

We come to the first set of results in this section, which we need to rule out the closed loop possibility of the solution in the global bifurcation analysis.  It shows that any nontrivial, even, periodic $C^1$ solution, which is nonincreasing on $(0, \pi)$, must be strictly decreasing and less than $\mu$ on $(0, \pi)$.  Thus, we conclude that $\max\phi = \phi(0) = \mu$.
\begin{lemma} \label{Lem:phi'<0:phi<mu}
	Any nontrivial, $2\pi$-periodic, even solution $\phi \in C^1(\T)$ of \eqref{main u eqn} which is nonincreasing on $(0,\pi)$ satisfies
	\[
	\phi'(x) < 0 \qquad \text{and} \qquad  \phi(x) < \mu \qquad \text{on } (0,\pi).
	\]
	For such a solution, one necessarily has $\mu > 0$.  Moreover, if $\phi \in C^2(\T)$, then $\phi''(0) < 0$.  
\end{lemma}
\begin{proof}
	Since $\phi \in C^1(\T)$, we can take the derivative of equation \eqref{eqv u eqn} to obain
	\[
	(\mu - \phi)\phi'(x) = L \phi'(x).
	\]
	Since $\phi$ is nonincreasing on $(0,\pi)$, we have $\phi'(x) \le 0$ on $(0,\pi)$.  We want to show that $L \phi'(x) < 0$ for $x \in (0,\pi)$.  Since $\phi'$ is odd, nontrivial, and nonpositive on $(0,\pi)$, Lemma~\ref{Lemma:L:monotone} implies $L \phi'(x) < 0$ on $(0,\pi)$.  Thus, we have 
	\[
	(\mu - \phi)\phi'(x) < 0, 
	\]
	and hence we have shown that $\phi'(x) < 0$ and $\phi(x) < \mu$ on $(0,\pi)$.  On the other hand, by Lemma~\ref{Lem:maxPhi:minPhi:estimate}, we know $0 \le \phi(0) = \phi(x_M)$, so that $\mu > 0$.  
	
	To prove the second part, assume that $\phi \in C^2(\T)$.  Differentiating twice equation~\eqref{eqv u eqn} gives
	\[
	(\mu - \phi)\phi'' = L \phi'' + (\phi')^2.
	\]
	Then evaluating this equality by $x = 0$ and using the evenness of $K_P$ and $\phi''$, we compute
	\begin{align*}
	(\mu - \phi)\phi''(0) &= \int_{-\pi}^\pi K_P(y) \phi''(y) \dm y \\
	&= 2 \int_0^\pi K_P(y) \phi''(y) \dm y \\
	&= 2 \int_0^\epsilon K_P(y) \phi''(y) \dm y + 2 \int_\epsilon^\pi K_P(y) \phi''(y) \dm y \\
	&= 2 \int_0^\epsilon K_P(y) \phi''(y) \dm y + 2 [K_P(y) \phi'(y)]|_{y = \epsilon}^{y = \pi} - 2 \int_\epsilon^\pi K_P'(y) \phi'(y) \dm y.
	\end{align*}
	Since $\phi$ is in $C^2$ and $K$ is integrable, the first integral vanishes as $\epsilon \to 0$, so does the boundary term $K_P(\epsilon) \phi'(\epsilon)$.  The term $K_P(\pi) \phi'(\pi)$ also vanishes since $\phi'(\pi) = 0$.  By Theorem~\ref{Thm:K-Properties} and what we just proved, both $K_P'$ and $\phi'$ are strictly negative on $(0,\pi)$, and hence we have
	\[
	(\mu - \phi) \phi''(0) = -2 \lim_{\epsilon \to 0^+} \int_\epsilon^\pi K_P'(y) \phi'(y) \dm y < 0,
	\]
	which shows that $\phi''(0) < 0$.	
\end{proof}
In both papers \cite{Bruell_Dhara2018,Ehrnstrom_Wahlen2019}, a solution to the nonlocal equation is smooth when it is below its maximum value.  In the next theorem, we show that the same conclusion is reached when the solution approaches the maximum value $\phi(0) = \mu$ from below.  Here we only rely on the boundedness of the solution.  
\begin{thm} \label{Thm:phi:smooth}
	Let $\phi \le \mu$ be a bounded solution of \eqref{eqv u eqn}.  Then:
	\begin{enumerate}[label=(\roman*)]
		\item If $\phi < \mu$ uniformly on $\T$, then $\phi \in C^\infty(\R)$.
		\item $\phi$ is smooth on any open set where $\phi < \mu$.  
	\end{enumerate}
\end{thm}
\begin{proof}
	To show part $(i)$, let $\phi < \mu$ uniformly on $\T$ be a bounded solution to equation~\eqref{eqv u eqn}.  We know from Section~\ref{Section:Setting:K*acts:period} that $L$ maps $\mathcal{C}^s(\T)$ into $\mathcal{C}^{s + \alpha}(\T)$ for any $s \in \R$.  In particular, $L$ maps $L^\infty(\T) \subset \mathcal{C}^0(\T)$ into $\mathcal{C}^\alpha(\T)$.  Moreover, if $s > 0$, the Nemytskii operator
	\[
	f \mapsto \mu - \sqrt{\frac{1}{2} \mu^2 - f}
	\]
	maps $L^\infty(\T) \cap B^s_{\infty,\infty}(\T)$ into itself for $f < \frac{1}{2} \mu^2$ and $s > 0$; see \cite[Theorem 2.87]{Bahouri_Chemin_Danchin2011}.  Since $\phi < \mu$, equation \eqref{eqv u eqn} gives
	\[
	L \phi < \frac{1}{2} \mu^2,
	\]
	and hence, we obtain the mapping
	\begin{equation} \label{Lphi:compose:map}
	\left[ L \phi \mapsto \sqrt{\frac{1}{2} \mu^2 - L \phi} \right] \circ \left[ \phi \mapsto L \phi \right] : L^\infty(\T) \cap B_{\infty,\infty}^s(\T) \to B_{\infty,\infty}^{s + \alpha}(\T)
	\end{equation}
	for all $s \ge 0$.  Finally, equation \eqref{eqv u eqn} yields
	\[
	\phi = \mu - \sqrt{\mu^2 - 2 L \phi}
	\]
	by the assumption $\phi < \mu$.  We also note that the mapping $x \mapsto \sqrt{x}$ is real analytic for $x > 0$, and hence an iterative argument in $s$ shows that $\phi \in C^\infty(\T)$.  To show that $\phi \in C^\infty(\R)$, we note that if $\phi$ is a periodic solution of \eqref{eqv u eqn}, its translation $\phi_h := \phi(\cdot + h)$ for any $h \in \R$ is also periodic.  However, the previous analysis implies that $\phi_h \in C^\infty(\T)$ for any $h \in \R$, so we conclude that $\phi \in C^\infty(\R)$.
	
	To show part $(ii)$, let $\phi \in L^\infty(\R)$, and assume that $\phi \in \mathcal{C}_\mathrm{loc}^s(U)$ for some open set $U$ and any $s \in \R$.  Let $\varphi \in C^\infty_0(U)$ be a smooth function compactly supported in $U$, and $\widetilde{\varphi} \in C^\infty_0(U)$ be a cut-off function with $\widetilde{\varphi} = 1$ in a neighborhood $V \Subset U$ of $\supp \varphi$.  Then
	\[
	\varphi L\phi = \varphi L(\widetilde{\varphi} \phi) + \varphi L((1 - \widetilde{\varphi})\phi).
	\]
	It is easy to see that $\varphi L(\widetilde{\varphi}\phi) \in \mathcal{C}_\mathrm{loc}^{s + \alpha}(U)$ and $\varphi L((1 - \widetilde{\varphi})\phi)$ vanishes by construction.  Thus, $L\phi \in \mathcal{C}_\mathrm{loc}^{s + \alpha}(U)$.  Using the same iteration argument in $s$ as in the proof of part $(i)$ shows that $\phi$ is smooth on $U$.  
\end{proof}

As a motivation, we shall see that when the order of the dispersive term is strictly smaller than $-1$, any decrease of the order does not affect the regularity of the wave solution.  
\begin{thm} \label{Thm:cannot-C1}
	Let $\phi \le \mu$ be an even solution of \eqref{eqv u eqn}, which is nonincreasing on $[0,\pi]$, and attains its maximum at $\phi(0) = \mu$.  Then $\phi$ cannot belong to the class $C^1(\T)$.  
\end{thm}
\begin{proof}
	By contradiction, suppose that $\phi$ is in $C^1(\T)$.  Since $K'_P$ is integrable by Theorem~\ref{Thm:K-Properties}, equation~\eqref{eqv u eqn} gives
	\[
	\norm[\infty]{\frac{\dm}{\dx} (\mu - \phi)^2} \le 2 \norm[1]{K'} \norm[\infty]{\phi}.
	\]
	Thus, $(\mu - \phi)^2$ is twice continuously differentiable.  Then fixing $|x| \ll 1$ and applying the Taylor theorem, we have
	\begin{equation} \label{tmp:Taylor:(mu-phi)^2}
	(\mu - \phi)^2(x) = (\mu - \phi)^2(0) + [(\mu - \phi)^2]'(0)x + \frac{1}{2} [(\mu - \phi)^2]''(\xi) x^2
	\end{equation}
	for some $\xi$ between $0$ and $x$.  Differentiating equation~\eqref{eqv u eqn} twice gives
	\[
	\frac{1}{2}[(\mu - \phi)^2]''(\xi) = -K' * \phi'(\xi),
	\]
	which substitutes into the expression \eqref{tmp:Taylor:(mu-phi)^2} yields
	\[
	(\mu - \phi)^2(x) = -K' * \phi'(\xi) x^2.
	\]
	On the other hand, since $K_P$ is even, nonincreasing on $(0,\pi)$ by Theorem~\ref{Thm:K-Properties}, and $\phi$ is nonincreasing on $(0, \pi)$, we obtain
	\[
	-K' * \phi'(0) = -\int_{-\pi}^\pi K_P'(-y) \phi'(y) \dm y = 2 \int_0^\pi K_P'(y) \phi'(y) \dm y = C > 0
	\]
	for some constant $C > 0$.  Both $K$ and $\phi$ are continuous, so there exists a constant $C_0 > 0$ such that 
	\[
	-K' * \phi'(\xi) \ge C_0
	\]
	for any $\xi$ between $0$ and $x$.  Using this bound and equation \eqref{tmp:Taylor:(mu-phi)^2}, we have the estimate
	\[
	(\mu - \phi)^2(x) \ge C_0 x^2,
	\]
	which is equivalent to
	\[
	\frac{\mu - \phi(x)}{|x|} \gtrsim 1.
	\]
	for all $|x| \ll 1$.  Taking the limit when $x \to 0$ leads to contradiction to the fact $\phi'(0)=0$.
\end{proof}

In the next theorem, we investigate the regularity of a solution when it touches $\mu$ from below.  It was shown that the solution has $\frac{1}{2}$-H\"older regularity at $x=0$ for an inhomogeneous multiplier of order greater than $-1$ \cite{Ehrnstrom_Wahlen2019}, and Lipschitz continuous at $x=0$ for a homogeneous multiplier of order less than $-1$ \cite{Bruell_Dhara2018}.  For our case with inhomogeneous symbol of order less than $-1$, we, indeed, obtain a Lipschitz regularity at $x=0$.
\begin{thm} \label{Thm:Soln:Phi:regularity:at:0}
	Let $\phi \le \mu$ be a nontrivial, $2\pi$-periodic, even solution of \eqref{eqv u eqn}, which is nonincreasing on $[0,\pi]$.  If $\phi$ attains its maximum at $\phi(0)=\mu$, then the following holds:
	\begin{enumerate}[label=(\roman*)]
		\item $\phi \in C^\infty(\T \backslash \{0\})$ and $\phi$ is strictly decreasing on $(0,\pi)$.
		\item $\phi \in C^{1-}(\T)$, that is $\phi$ is Lipschitz continuous.
		\item $\phi$ is precisely Lipschitz continuous at $x=0$, that is, there exist constants $c_1, c_2 > 0$ such that
		\[
		c_1 |x| \le \mu - \phi(x) \le c_2 |x|
		\]
		for $|x| \ll 1$.
	\end{enumerate}
\end{thm}
\begin{proof}
	Assume that $\phi \le \mu$ is an even solution of equation \eqref{eqv u eqn}, which is nonincreasing on $(0, \pi)$ and attains its maximum at $\phi(0) = \mu$.  
	\begin{enumerate}[label=$(\roman*)$]
	\item Let $x \in (0, \pi)$ and $h \in (-\pi,0)$.  Then we rewrite equation~\eqref{eqv u eqn} as
	\begin{align*}
	\frac{1}{2} (2\mu - \phi(x+h) - \phi(x-h)) & (\phi(x+h) - \phi(x-h)) \\
	&= K*\phi(x+h) - K*\phi(x-h)
	\end{align*}
	We want to show that the right-hand side is strictly negative.  In fact, for any $x \in (0, \pi)$ and $h \in (-\pi, 0)$, since $\phi$ and $K_P$ are even and periodic, we obtain
	\begin{align*}
	K*\phi(x+&h) - K*\phi(x-h) \\
	&= \int_0^\pi (K_P(x + y) - K_P(x - y)) (\phi(y-h) - \phi(y+h)) \dm y.
	\end{align*}
	By similar arguments as in the proof of Lemma~\ref{Lemma:L:monotone}, we know that $K_P(x + y) - K_P(x - y) < 0$ for $x, y \in (0,\pi)$,  We further have $\phi(y - h) - \phi(y + h) \le 0$ for $y \in (0,\pi)$ and $h \in (-\pi, 0)$ by the assumption that $\phi$ is even and nonincreasing on $(0,\pi)$.  Therefore, the integrand is nonpositive.  Since $\phi$ is a nontrivial solution and $K_P$ is not a constant, we conclude that
	\begin{equation} \label{K*phi(y+h)-K*phi(y-h)<0}
	K*\phi(x+h) - K*\phi(x-h) > 0
	\end{equation}
	for any $h \in (-\pi,0)$.  Hence, $K*\phi(x + h) = K*\phi(x - h)$ if and only if $\phi(x + h) = \phi(x - h)$.  Then inequality \eqref{K*phi(y+h)-K*phi(y-h)<0} implies
	\[
	\phi(x+h) > \phi(x-h) \qquad \text{for any} \quad h \in (-\pi,0).
	\]
	Thus, $\phi$ is strictly decreasing on $(0,\pi)$.  In view of Therem~\ref{Thm:phi:smooth}, $\phi$ is smooth on $\T \backslash \{0\}$.
	
	\item To prove the Lipschitz regularity at the maximum point, we make use of a bootstrap argument.  By contradiction, assume that the solution $\phi < \mu$ is not Lipschitz continuous.  Suppose that $\phi$ is only a bounded function.  Then recalling that $L$ maps $L^\infty \subset \mathcal{C}^0(\T)$ into $\mathcal{C}^\alpha(\T)$ for $\alpha > 1$ and the expression
	\[
	\frac{1}{2} (2\mu - \phi(x) - \phi(y)) (\phi(x) - \phi(y)) = L \phi(x) - L \phi(y),
	\]
	we have the estimate
	\[
	\frac{1}{2} (\phi(x) - \phi(y))^2 \le |L \phi(x) - L \phi(y) | \simeq |x-y|.
	\]
	Thus, it is straightforward to see that $\phi$ is $\frac{1}{2}$-H\"older continuous.  
	
	Next, evaluating equation \eqref{eqv u eqn} at $x = 0$ gives
	\[
		-\frac{1}{2} \mu^2 + K*\phi(0) = 0,
	\]
	and then subtracting from equation \eqref{eqv u eqn}, we obtain
	\[
	\frac{1}{2} (\mu - \phi)^2(x) = K*\phi(0) - K*\phi(x)
	\]
	for $x \in (0, \pi)$.  Since $\phi$ is smooth on $\T \backslash \{0\}$ by Theorem~\ref{Thm:phi:smooth}, differentiating the above equality yields
	\begin{equation*} \label{tmp:(mu-phi)phi':equality}
	(\mu - \phi) \phi'(x) = (K*\phi)'(x) - (K*\phi)'(0),
	\end{equation*}
	where we are using the fact that $(K*\phi)'(0) = 0$ and $(K*\phi(0))' = 0$.  If $\phi$ is $\frac{1}{2}$-H\"older continuous, then $K*\phi \in \mathcal{C}^{\frac{1}{2} + \alpha}(\T)$.  Since $\alpha > 1$, we gain at least some H\"older regularity for $(K*\phi)'$.  Thus, the right-hand side of expression~\eqref{tmp:(mu-phi)phi':equality} can be estimated by a constant multiple of $|x|^\beta$ for some $\beta \in \left( \frac{1}{2}, 1 \right]$, and hence,
	\begin{equation} \label{mu-phi:phi'<|x|^beta}
	(\mu - \phi) \phi'(x) \lesssim |x|^\beta.
	\end{equation}
	By assumption that $\phi$ is not Lipschitz continuous at $x=0$, the above estimate guarantees that $\phi$ is at least $\beta$-H\"older continuous, where $\beta > \frac{1}{2}$.  We aim to bootstrap this argument to obtain Lipschitz regularity of $\phi$ at $x = 0$.  If $\frac{1}{2} + \alpha > 2$, we use that $K*\phi \in \mathcal{C}^{\frac{1}{2} + \alpha}(\T) \subset C^2(\T)$, which guarantees that its derivative is at least Lipschitz continuous ($\beta = 1$ in expression \eqref{mu-phi:phi'<|x|^beta}), and hence, $\phi$ is Lipschitz continuous.  On the other hand, if $\frac{1}{2} + \alpha \le 2$, then $\phi$ is $\gamma$-H\"older continuous for some $\gamma > \beta$.  We then repeat the argument finitely many times to yield that $\phi$ is indeed Lipschitz continuous at $x=0$, that is
	\begin{equation} \label{upper-bound:mu-phi}
	\mu - \phi(x) \lesssim |x|, \qquad \mathrm{for} \quad |x| \ll 1.
	\end{equation}

	\item From the upper bound \eqref{upper-bound:mu-phi}, it remains to show that
	\begin{equation} \label{tmp:upper-bound:mu-phi}
	(\mu - \phi)\phi'(x) \gtrsim |x|
	\end{equation}
	for $|x| \ll 1$.  Taking the derivative of equation~\eqref{eqv u eqn} and evaluating at any $\xi \in (0,\pi)$ we have
	\[
	(\mu - \phi) \phi'(\xi) = (K*\phi)'(\xi) = \int_0^\pi (K_P(\xi - y) - K_P(\xi + y)) \phi'(y) \dm y.
	\]
	On the other hand, evaluating the upper bound established in \eqref{upper-bound:mu-phi} at $x = \xi \in (0, \pi)$ gives $\mu - \phi(\xi) \lesssim |\xi|$.  Therefore, dividing the above equation by $(\mu - \phi)(\xi) > 0$ to obtain the inequality
	\begin{equation} \label{tmp:phi':ineq}
	\phi'(\xi) \gtrsim \int_0^\pi \frac{K_P(\xi - y) - K_P(\xi + y)}{|\xi|} \phi'(y) \dm y.
	\end{equation}
	Our aim is to show that $\liminf_{\xi \to 0^+} \phi'(\xi)$ is strictly bounded away from $0$.  We compute
	\begin{align*}
	&\lim_{\xi \to 0^+} \frac{K_P(\xi - y) - K_P(\xi + y)}{|\xi|} \\
	&= \lim_{\xi \to 0^+} \left( \frac{K_P(y - \xi) - K_P(y)}{\xi} + \frac{K_P(y) - K_P(y + \xi)}{\xi} \right) \frac{\xi}{|\xi|} = 2 K_P'(y)
	\end{align*}
	for any $y \in (0,\pi)$.  Then since $K'_P$ is integrable, taking the limit on both sides of the inequality~\eqref{tmp:phi':ineq} yields
	\begin{equation} \label{liminf:gtrsim:c}
	\liminf_{\xi \to 0} \phi'(\xi) \gtrsim 2 \int_0^\pi K_P'(y) \phi'(y) \dm y = c
	\end{equation}
	for some constant $c > 0$, since $\phi$ and $K_P$ are strictly decreasing on $(0,\pi)$.  Then for any $0 < x \ll 1$, applying the Mean Value Theorem, we have
	\[
	\frac{\phi(0) - \phi(x)}{x} = \phi'(z)
	\]
	for some $z \in (0, x)$.  Combing with the expression \eqref{liminf:gtrsim:c} gives our desired estimate~\eqref{tmp:upper-bound:mu-phi}, and hence, we conclude that $\phi$ is exactly Lipschitz continuous at $x = 0$.  \qedhere
	\end{enumerate}
\end{proof}

From Theorem \ref{Thm:phi:smooth} and \ref{Thm:Soln:Phi:regularity:at:0}, we conclude that any even, periodic solution $\phi \le \mu$ of equation \eqref{eqv u eqn}, which is monotone on half the period, is Lipschitz continuous.  Therefore, if such a solution occurs, it must be the highest, peaked wave.  Before closing this section, we provide a lemma to prove that one of the alternatives in the global bifurcation arguments in Section~\ref{Section:Global:bifurcation} does not occur.
\begin{lemma} \label{Lem:mu-phi:ge:lambdaP}
	Let $\phi \le \mu$ be an even solution of equation \eqref{eqv u eqn}, which is nonincreasing on $(0, \pi)$.  Then there exists a constant $\lambda = \lambda(\alpha) > 0$, depending only on the kernel $K$, such that
	\[
	\mu - \phi(\pi) \ge \lambda \pi.
	\]
\end{lemma}

\begin{proof}
	Choose any $x \in \left[ \frac{\pi}{4}, \frac{3\pi}{4} \right]$.  Then by the evenness of $K$ and the fact that $\phi$ is nonincreasing on $(0, \pi)$, we have the estimate
	\begin{equation} \begin{split} \label{tmp:lemma4.6}
	(\mu - \phi(\pi)) \phi'(x) &\le (\mu - \phi(x)) \phi'(x) \\
	&= \int_0^\pi (K_P(x - y) - K_P(x + y)) \phi'(y) \dm y \\
	&\le \int_{\pi/4}^{3\pi/4} (K_P(x - y) - K_P(x + y)) \phi'(y) \dm y,
	\end{split} \end{equation}
	since the integrand is nonpositive.  Notice that $K_P(x - y) - K_P(x + y) > 0$ for $x, y \in (0, \pi)$, and there exists a constant $\lambda = \lambda(\alpha) > 0$, depending only on the kernel $K$, such that
	\[
	K(x - y) - K(x + y) \ge 2\lambda \qquad \text{for all} \quad x, y \in \left( \frac{\pi}{4}, \frac{3\pi}{4} \right).
	\]
	Thus, integrating the inequality \eqref{tmp:lemma4.6} with respect to $x$ over $\left( \frac{\pi}{4}, \frac{3\pi}{4} \right)$ yields
	\begin{align*}
	(\mu - \phi(\pi)) &\left( \phi \left( \frac{3\pi}{4} \right) - \phi \left( \frac{\pi}{4} \right) \right) \\
	&\le \int_{\pi/4}^{3\pi/4} \left( \int_{\pi/4}^{3\pi/4} K(x - y) - K(x + y) \dx \right) \phi'(y) \dm y \\
	&\le \lambda \pi \left( \phi \left( \frac{3\pi}{4} \right) - \phi \left( \frac{\pi}{4} \right) \right).
	\end{align*}
	From Theorem~\ref{Thm:Soln:Phi:regularity:at:0}, we know that $\phi$ is strictly decreasing on $\left( \frac{\pi}{4}, \frac{3\pi}{4} \right)$, and hence we can divide the above inequality by the quantity $\phi \left( \frac{3\pi}{4} \right) - \phi \left( \frac{\pi}{4} \right) < 0$ to obtain the claim.  
\end{proof}

\section{Global bifurcation and conclusion of the main theorem}
\label{Section:Global:bifurcation}

In the last section, we proved the existence of nontrivial, highest, even, $2\pi$-periodic solutions of equation \eqref{main u eqn} using an analytic bifurcation technique.  In fact, both local and global bifurcation studies on nonlocal equations have been investigated intensively.  For instance, the existence of smooth, small-amplitude, periodic traveling-wave solutions to Whitham equation was established by Ehrnstr\"om and Kalisch in \cite{Ehrnstrom_Kalisch2009} using Crandall--Rabinowitz local bifurcation theorem.  The authors also investigated numerically a global branch of solutions approaching a highest, cusped, traveling-wave solution.  An analytic proof for the latter fact was provided in \cite{Ehrnstrom_Groves_Wahlen2012} using a variational approach.  Recently, Truong, Wahl\'en, and Wheeler \cite{Truong_Wahlen_Wheeler2020} attacked a similar problem using the center manifold theorem for the Whitham equation.  

In this paper, we make use of the same arguments in \cite{Ehrnstrom_Kalisch2013} which gave a general functional-analytic framework for bifurcation theory to Whitham equation.  Our aim is to extend the local bifurcation branch found by the analytic version of Crandall--Rabinowitz theorem to a global one, and then characterize the end of this bifurcation curve.  We will show that the global bifurcation curve reaches a limiting highest wave $\phi$, which is even, strictly decreasing on $(0,\pi)$ and attains its maximum $\phi(0) = \mu$.  By Theorem \ref{Thm:Soln:Phi:regularity:at:0}, the highest wave is a peaked traveling-wave solution of
\[
u_t + L u_x + u u_x = 0 \qquad \text{for } \quad \alpha > 1.
\]
We use the subscript $X_\mathrm{even}$ for the restriction of a Banach space $X$ to its subset of even functions.  Let $\beta \in (1,2)$ and set
\[
F: C^\beta_\mathrm{even}(\T) \times \R_+ \to C^\beta_\mathrm{even}(\T),
\]
where
\begin{equation} \label{F defn}
F(\phi,\mu) := \mu\phi - L\phi - \frac{1}{2} \phi^2, \qquad (\phi,\mu) \in C^\beta_\mathrm{even}(\T) \times \R_+.
\end{equation}
Then $F(\phi,\mu)=0$ if and only if $\phi$ is an even $C^\beta(\T)$-solution of \eqref{main u eqn} corresponding to the wave speed $\mu \in \R_+$.  We have the first local bifurcation result as follow.

\begin{thm}[Local bifurcation] \label{Thm:local:bifurcation}
	For each integer $k \ge 1$, the point $(0,\mu_k^*)$, where $\mu_k^* = (1 + k^2)^{-\alpha/2}$ is a bifurcation point.  More precisely, there exits $\epsilon_0 > 0$ and an analytic curve through $(0,\mu_k^*)$,
	\[
	\{ (\phi_k(\epsilon), \mu_k(\epsilon)) : |\epsilon| < \epsilon_0 \} \subset C^\beta_\mathrm{even}(\T) \times \R_+,
	\]
	of nontrivial, $\frac{2\pi}{k}$-periodic, even solutions of \eqref{F defn} with $\mu_k(0) = \mu_k^*$ and
	\[
	D_\epsilon \phi_k(0) = \phi_k^*(x) = \cos(xk).
	\]
	In a neighborhood of the bifurcation point $(0, \mu_k^*)$ these are all the nontrivial solutions of $F(\phi,\mu) = 0$ in $C^\beta_\mathrm{even}(\T) \times \R_+$.
\end{thm}
\begin{proof}
	We will prove the result using the analytic version of the Crandall--Rabinowitz Theorem~\cite[Theorem 8.4.1]{Buffoni_Toland2003}.  It is clear that $F(0,\mu)=0$ for any $\mu\in\R_+$.  We are looking for $2\pi$-periodic, even, nontrivial solutions bifurcating from the line $\{(0,\mu): \mu\in\R\}$ of trivial solutions.  The wave speed $\mu > 0$ shall be the bifurcation parameter.  The linearization of $F$ around the trivial solution $(\phi=0,\mu)$ is given by
	\[
	D_\phi F(0,\mu): C^\beta_\mathrm{even}(\T) \to C^\beta_\mathrm{even}(\T), \qquad 
	\phi \mapsto (\mu \, \mathrm{Id} - L)\phi.
	\]
	Then from Section~\ref{Section:Setting:K*acts:period}, we know that $L:C^\beta_\mathrm{even}(\T) \to C^{\beta+\alpha}_\mathrm{even}(\T)$ is parity preserving and a smoothing operator, which implies that it is compact on $C^\beta_\mathrm{even}(\T)$.  Hence, $D_\phi F(0,\mu)$ is a compact perturbation of an isomorphism, and therefore constitutes a Fredholm operator of index $0$.  The nontrivial kernel of $D_\phi F(0,\mu)$ is spanned by functions $\psi \in C^\beta_\mathrm{even}(\T)$ satisfying
	\[
	\fourier{\psi}(k) (\mu - (1+k^2)^{-\alpha/2}) = 0
	\]
	for all $k$.  For $\mu \in (0,1]$, we see that $\supp \psi \subseteq \{ \pm (\mu^{-2/\alpha}-1)^{1/2} \}$.  Therefore, the kernel of $D_\phi F(0,\mu)$ is one-dimensional if and only if $\mu = \mu_k^* := (1 + k^2)^{-\alpha/2}$ for some $k \in \mathbb{Z}$, in which case it is given by
	\[
	\ker D_\phi F(0,\mu) = \mathrm{span} \{\phi_k^*\} \qquad \text{with} \quad \phi_k^*(x) := \cos(xk).
	\]
	We also note that $\mu_k^*$ are all simple eigenvalues of $L$.  Finally, we observe that
	\[
	D_{\phi\mu} F(0, \mu_k^*)\phi_k^* = \phi_k^* 
	\]
	is not in the range of $D_\phi F(0, \mu_k^*)$, which means the tranversality condition is satisfied. Thus, the assumptions of the Crandall--Rabinowitz theorem are fulfilled.  
\end{proof}

Next, our aim is to extend the local bifurcation branch found in Theorem~\ref{Thm:local:bifurcation} to a global continuum of solutions of $F(\phi,\mu)=0$.  Recalling $\beta \in (1, 2)$, set
\[
S := \{ (\phi,\mu) \in U: F(\phi,\mu) = 0 \},
\]
where the admissible set $U$ is given by
\[
U := \{ (\phi,\mu) \in C^\beta_\mathrm{even}(\T) \times \R_+ : \phi < \mu \}.
\]
Then all bounded solutions $\phi$ of equation~\eqref{eqv u eqn} satisfy $F(\phi,\mu) = 0$ for all $(\phi,\mu) \in S$.  We start the analysis by providing the $L^\infty$-bound for a solution.
\begin{lemma}[$L^\infty$ bound] \label{Lem:Phi:estimate}
	Let $\mu > 0$.  Then any bounded solution $\phi$ to equation \eqref{main u eqn} satisfies:
	\[
	\norm[\infty]{\phi} \le 2(\mu + \norm[1]{K}).
	\]
\end{lemma}
\begin{proof}
	From equation~\eqref{eqv u eqn}, we have the estimate
	\[
	\norm[\infty]{\phi}^2 \le 2 (\mu + \norm[1]{K}) \norm[\infty]{\phi},
	\]
	where we have used the fact that the kernel $K$ is integrable.  Therefore, $\phi \equiv 0$, which holds trivially, or the desired estimate follows by dividing by $\norm[\infty]{\phi}$.  
\end{proof}

Before extending the local branches globally, we need two helping lemmas.
\begin{lemma} \label{Lem:D_phi:F:Fredholm}
	The Frech\'et derivative $D_\phi F(\phi,\mu)$ is a Fredholm operator of index $0$ for all $(\phi, \mu) \in U$.
\end{lemma}
\begin{proof}
	We have
	\[
	D_\phi F(\phi,\mu) = (\mu - \phi)\mathrm{Id} - L
	\]
	for any given $(\phi,\mu) \in U$.  Since $(\mu-\phi)\mathrm{Id}$ is an isomorphism on $C^\beta(\T)$ and $L$ is compact on $C^\beta(\T)$, the operator $D_\phi F(\phi,\mu)$ is Fredholm.  From the proof of Theorem~\ref{Thm:local:bifurcation}, we know that $D_\phi F(0, \mu)$ has Fredholm index 0, and so is $D_\phi F(\phi,\mu)$ due to the fact that the index is continuous.  
\end{proof}

\begin{lemma} \label{Lem:bdd:closed:subsetS:compact}
	For any $(\phi,\mu) \in S$, the function $\phi$ is smooth, and any bounded and closed subset of $S$ is compact in $C^\beta_\mathrm{even}(\T) \times \R_+$.
\end{lemma}
\begin{proof}
	If $(\phi,\mu) \in S$, then $\phi < \mu$, and we write equation~\eqref{main u eqn} in the form
	\[
	\phi = \mu - \sqrt{\mu^2 - 2L\phi} =: \widetilde F(\phi,\mu).
	\]
	The function $\widetilde F$ is a bounded and linear mapping from $U$ into  $C^{\beta+\alpha}_\mathrm{even}(\T)$.  Moreover, by Theorem~\ref{Thm:phi:smooth}, we know that $\phi$ is smooth.
	
	Let $A \subset S \subset U$ be a bounded and closed set.  Then $\widetilde F(A) = \{ \phi: (\phi,\mu) \in A \}$ is relatively compact in $C^\beta_\mathrm{even}(\T)$.  Since $A$ is closed, any sequence $\{(\phi_n,\mu_n)\}_{n \in \mathbb{N}}$ has a convergent subsequence in $A$ by Arzela--Ascoli's lemma.  We conclude that $A$ is compact in $C^\beta_\mathrm{even}(\T) \times \R_+$.  
\end{proof}

According to \cite[Theorem 9.1.1]{Buffoni_Toland2003}, Lemma \ref{Lem:D_phi:F:Fredholm} and \ref{Lem:bdd:closed:subsetS:compact} allow us to extend the local branches found in Theorem~\ref{Thm:local:bifurcation} to global curves once we establish that any of the derivatives $\mu(\epsilon)$ is not identically zero for $0 < \epsilon \ll 1$.  However, the latter claim is an immediate consequence of Theorem~\ref{Thm:Global:bifurcation:formula} below.  This theorem is an adaptation of \cite[Theorem 4.4]{Ehrnstrom_Kalisch2013} with $U$ and $S$ as above.  

\begin{thm} [Global bifurcation] \label{Thm:Global:bifurcation}
	The local bifurcation curve $s \mapsto (\phi_k(s), \mu_k(s))$ from Theorem~\ref{Thm:local:bifurcation} of solutions of equation \eqref{F defn} extends to a global continuous curve of solutions $\mathfrak{G}: \R_+ \to S$, that allows a local real-analytic reparameterization around each $s > 0$.  One of the following alternatives holds:
	\begin{enumerate}[label=(\roman*)]
		\item $\| (\phi_k(s), \mu_k(s)) \|_{C^\beta(\T) \times \R_+} \to \infty$ as $s \to \infty$.
		\item There exists a subsequence $(\phi_k(s_n), \mu_k(s_n))_{n \in \N}$ such that the pair $(\phi_k(s_n), \linebreak[1] \mu_k(s_n))$ approaches the boundary of $S$ as $n \to \infty$.
		\item The function $s \mapsto (\phi_k(s), \mu_k(s))$ is (finitely) periodic.
	\end{enumerate}
\end{thm}
We apply the Lyapunov--Schmidt reduction, in order to establish the bifurcation formulas.  Let $k \in \mathbb{N}$ be a fixed number and set
\[
M := \mathrm{span} \{ \cos(xl): l \ne k \}, \qquad 
N := \ker D_\phi F(0,\mu_k^*) = \mathrm{span}\{\phi_k^*\}.
\]
Then $C^\beta_\mathrm{even}(\T) = M \oplus N$ and a continuous projection onto the one-dimensional space $N$ is given by
\[
\Pi \phi = \langle \phi,\phi_k^* \rangle_{L^2} \phi_k^*,
\]
where $\langle \cdot,\cdot \rangle_{L^2}$ denotes the inner product in $L^2(\T)$.  Let us recall the Lyapunov--Schmidt reduction theorem from \cite[Theorem I.2.3]{Kielhofer2012}.
\begin{thm}[Lyapunov--Schmidt reduction]
	There exists a neighborhood $\mathcal{O} \times Y \subset U$ of $(0,\mu_k^*)$ such that the problem
	\begin{equation} \label{Lyap:Schmidt:inf_dimen:eqn}
	F(\phi,\mu) = 0 \qquad \mathrm{for} \quad (\phi,\mu) \in \mathcal{O} \times Y
	\end{equation}
	is equivalent to the finite-dimensional problem
	\begin{equation} \label{Lyap:Schidt:equiv:eqn}
	\Phi(\epsilon \phi_k^*, \mu) := \Pi F(\epsilon\phi_k^* + \psi(\epsilon\phi_k^*, \mu), \mu) = 0
	\end{equation}
	for functions $\psi \in C^\infty(\mathcal{O}_N \times Y, M)$ and $\mathcal{O}_N \subset N$ an open neighborhood of the zero function in $N$.  One has that $\Phi(0,\mu_k^*) = 0$, $\psi(0,\mu_k^*)=0$, $D_\phi \psi(0,\mu_k^*)=0$, and solving problem \eqref{Lyap:Schidt:equiv:eqn} provides a solution
	\[
	\phi = \epsilon \phi_k^* + \psi(\epsilon \phi_k^*, \mu)
	\]
	of the infinite-dimensional problem \eqref{Lyap:Schmidt:inf_dimen:eqn}.
\end{thm}

The next theorem gives bifurcation formulas for the curve, which also justifies our arguments to extend the local curves globally.  
\begin{thm}[Bifurcation formulas] \label{Thm:Global:bifurcation:formula}
	The bifurcation curve found in Theorem~\ref{Thm:Global:bifurcation} satisfies
	\begin{equation} \label{bifurcation:formula:phi_k}
	\phi_k(\epsilon) = \epsilon \cos(kx) + \frac{\epsilon^2}{4} \left( \frac{1}{m(k) - m(0)} + \frac{1}{m(k) - m(2k)} \cos(2kx) \right) + O(\epsilon^3)
	\end{equation}
	and
	\begin{equation} \label{bifurcation:formula:mu_k}
	\mu_k(\epsilon) = m(k) + \frac{\epsilon^2}{4} \left( \frac{1}{m(k) - m(0)} + \frac{1}{2(m(k) - m(2k))} \right) + O(\epsilon^3)
	\end{equation}
	in $C^\beta_\mathrm{even}(\T) \times \R_+$ as $\epsilon \to 0$, where $m(k) = (1 + k^2)^{-\alpha/2}$, $\alpha > 1$.  In particular, $\ddot\mu(0) > 0$ for any $k \ge 1$, that is, the local bifurcation in Theorem~\ref{Thm:local:bifurcation} describes a supercritical pitchfork bifurcation.
\end{thm}
\begin{proof}
	For $(\phi_k(s), \mu_k(s))_{k \in \N}$ along the analytic local bifurcation curve in Theorem~\ref{Thm:local:bifurcation}, we fix $k \in \N$ and suppress the subscript $k$ to lighten our notation.  We first show that
	\[
	\mu(s) = \mu(-s)
	\]
	after a suitable choice of parameterization.  Since $\phi(s, \cdot)$ is even, it has a Fourier cosine representation
	\[
	\phi(s, \cdot) = \frac{1}{2} [\phi(s)]_0 + \sum_{j=1}^\infty [\phi(s)]_j \cos(j \cdot),
	\]
	where 
	\[
	[\phi(s)]_j = \frac{1}{\pi} \int_{-\pi}^\pi \phi(s, y) \cos(jy) \dm y
	\]
	for $j = 0, 1, 2, \dots$  Here we use a slightly different convention for the Fourier series compared to Section~\ref{Section:Setting:K*acts:period}.  Also, we choose the parameter in the local bifurcation curve so that $[\phi(s)]_1 = s$.  We observe that if $\phi(s)$ is an even and $2\pi$-periodic function, then $\phi(s, \cdot + \pi)$ is also an even and $2\pi$-periodic function, and its Fourier coefficient satisfies
	\[
	[\phi(s, \cdot + \pi)]_1 = \frac{1}{\pi} \int_{-\pi}^\pi \phi(s, y + \pi) \cos(y) \dm y = -[\phi(s)]_1.
	\]  
	Since $[\phi(s, \cdot + \pi)]_1 = -[\phi(s)]_1 = -s$, uniqueness of the bifurcation curve yields
	\[
	(\phi(s, \cdot + \pi), \mu(s)) = (\phi(-s, \cdot), \mu(-s)),
	\]
	which means $\mu(s) = \mu(-s)$.  Therefore, by the analyticity of $\mu(s)$, we can write
	\[
	\mu(s) = \sum_{n=0}^\infty \mu_{2n} s^{2n},
	\]
	where the sum is uniformly convergent in a neighborhood of $0$.  Similarly, $\phi(s)$ has an expansion
	\[
	\phi(s) = \sum_{n=1}^\infty \phi_n s^n
	\]
	with convergence in $C_\mathrm{even}^\beta(\T)$.  Substituting these two formulae into the main equation~\eqref{main u eqn} and equating coefficients of equal order in $s$, we obtain
	\begin{align}
	\label{tmp:1-order:s}
	L\phi_1 - \mu_0 \phi_1 &= 0, \\
	\label{tmp:2-order:s}
	L\phi_2 - \mu_0 \phi_2 &= -\frac{1}{2} \phi_1^2, \\
	\label{tmp:3-order:s}
	L\phi_3 - \mu_0 \phi_3 &= \mu_2 \phi_1 - \phi_1 \phi_2.
	\end{align}
	By Theorem~\ref{Thm:local:bifurcation}, $\phi_1 = \phi_k^*(x) = \cos(kx)$ and $\mu_0 = \mu_k^* = m(k)$, which implies expression \eqref{tmp:1-order:s} is satisfied.  Under the assumption that the right-hand sides of above equations lie in the range of the linear operators from the left-hand sides, the coefficients $\mu_n$ and $\phi_n$ can be determined by solving the corresponding equation.
	
	Recalling the formula for $L$ given by \eqref{L:formula} and expressing $\phi_2(x)$ as a Fourier series, equation~\eqref{tmp:2-order:s} gives
	\[
	\frac{1}{2\pi} \sum_{n \in \mathbb{Z}} m(n) \fourier\phi_2(n) \exp(inx) - \frac{m(k)}{2\pi} \sum_{n \in \mathbb{Z}} \fourier\phi_2(n) \exp(inx) = -\frac{1}{4} - \frac{1}{4} \cos(2kx),
	\]
	and then comparing coefficients of cosine functions, we obtain the relation
	\[
	\begin{dcases}
		\frac{1}{2\pi} m(0) \fourier\phi_2(0) - \frac{m(k)}{2\pi} \fourier\phi_2(0) &= -\frac{1}{4}, \\
		\frac{1}{\pi} m(2k) \fourier\phi_2(2k) - \frac{m(k)}{\pi} \fourier\phi_2(2k) &= -\frac{1}{4},
	\end{dcases}
	\]
	which yields
	\[
	\phi_2(x) = \frac{1}{4 (m(k) - m(0))} + \frac{1}{4 (m(k) - m(2k))} \cos(2kx).
	\]
	Next, the right-hand side of equation \eqref{tmp:3-order:s} becomes
	\[
	\left( \mu_2 - \frac{1}{4(m(k) - m(0))} - \frac{1}{8(m(k) - m(2k))} \right) \cos(kx) - \frac{1}{8(m(k) - m(2k))} \cos(3kx).
	\]
	Recalling the above parameterization $[\phi(s)]_1 = s$, which implies that $[\phi_n]_1 = 0$ for all $n \ge 2$, we find that
	\[
	\mu_2 = \frac{1}{4(m(k) - m(0))} + \frac{1}{8(m(k) - m(2k))} > 0.
	\]
	Thus, the bifurcation formulas \eqref{bifurcation:formula:phi_k} and \eqref{bifurcation:formula:mu_k} hold, and the local branch in Theorem \ref{Thm:local:bifurcation} is supercritical pitchfork bifurcation.  
\end{proof}

The remaining of the section is to devote showing that alternative $(iii)$ in Theorem~\ref{Thm:Global:bifurcation} is excluded, and both alternatives $(i)$ and $(ii)$ occur simultaneously as $s \to \infty$ along the bifurcation branch $\mathfrak{G}$.  This implies that the highest wave is reached as a limit of the global bifurcation curve.  
\begin{lemma} \label{Lem:global:bdd:sequence:has:convergent:subsequence}
	Any sequence of solutions $(\phi_k, \mu_k)_{k \in \N} \subset S$ to equation~\eqref{eqv u eqn} with $(\mu_k)_{k \in \N}$ bounded has a subsequence which converges uniformly to a solution $\phi$.  
\end{lemma}
\begin{proof}
	Recalling the estimate of $\phi$ in Lemma~\ref{Lem:Phi:estimate},
	\[
	\norm[\infty]{\phi} \le 2(\mu + \norm[1]{K}),
	\]
	we see that if $(\mu_k)_{k \in \N}$ is bounded, then $(\phi_k)_{k \in \N}$ is bounded.  
	Since $K$ is integrable and continuous on $\R$, Dominated Convergence Theorem allows us to conclude that $(L \phi_k)_{k \in \N}$ is equicontinuous.  Then Arzela--Ascoli's lemma implies that there exists a subsequence which converges uniformly to a solution $\phi$ of equation~\eqref{eqv u eqn}.  
\end{proof}

Let
\[
\mathcal{K}_k := \{ \phi \in C^\beta_\mathrm{even}(\T): \phi \text{ is } 2\pi/k \text{-periodic and nonincreasing in } (0, \pi/k)\}
\]
be a closed cone in $C^\beta(\T)$.  The next proposition exclude alternative $(iii)$ in Theorem~\ref{Thm:Global:bifurcation}.
\begin{prop} \label{Prop:Alternative(iii):not:occur}
	The solutions $\phi_k(s)$, $s > 0$ on the global bifurcation curve $\mathfrak{R}$ belong to $\mathcal{K}_k \backslash \{0\}$ and alternative $(iii)$ in Theorem~\ref{Thm:Global:bifurcation} does not occur.  In particular, the bifurcation curve $(\phi_k(s), \mu_k(s))$ has no intersection with the trivial solution line for any $s > 0$.
\end{prop}
\begin{proof}
	Due to \cite[Theorem 9.2.2]{Buffoni_Toland2003} the statement holds true if the following conditions are satisfied
	\begin{enumerate} [label = (\alph*)]
		\item $\mathcal{K}_k$ is a cone in a real Banach space.
		\item $(\phi_k(\epsilon), \mu_k(\epsilon)) \subset \mathcal{K}_k \times \R$ provided $\epsilon$ is small enough.
		\item If $\mu\in\R$ and $\phi \in \ker D_\phi F(0, \mu) \cap \mathcal{K}_k$, then $\phi = \gamma \phi^*$ for $\gamma \ge 0$ and $\mu = \mu_k^*$.
		\item Each nontrivial point on the bifurcation curve which also belongs to $\mathcal{K}_k \times \R$ is an interior point of $\mathcal{K}_k \times \R$ in $S$.
	\end{enumerate}
	Conditions (a), (b), (c) are satisfied because of the local bifurcation result in Theorem~\ref{Thm:local:bifurcation}, so it remains to verify condition (d).  Let $(\phi,\mu) \in \mathcal{K}_k \times \R$ be a nontrivial solution on the bifurcation curve found in Theorem~\ref{Thm:Global:bifurcation}.  By Theorem~\ref{Thm:phi:smooth}, $\phi$ is smooth and together with Lemma~\ref{Lem:phi'<0:phi<mu}, we have $\phi' < 0$ on $(0, \pi)$ and $\phi''(0) < 0$.  Choose a solution $\varphi$ lying within $|\delta| \ll 1$ small enough neighborhood in $C^\beta(\T)$ such that $\varphi < \mu$ and $\|\phi - \varphi\|_{C^\beta} < \delta$.  In view of the mapping~\eqref{Lphi:compose:map}, an iteration process on the regularity index yields that $\|\phi - \varphi\|_{C^2} < \widetilde\delta$, where $\widetilde\delta=\widetilde\delta(\delta) > 0$ can be made arbitrarily small by choosing $\delta$ small enough.  It follows that for $\delta$ small enough, $\varphi < \mu$ is a smooth, even, nonincreasing on $(0, \pi)$ solution, and hence $(\phi,\mu)$ belongs to the interior of $\mathcal{K}_k \times \R$ in $S$, which concludes the proof.
\end{proof}

\begin{remark} \label{Rmk:0<mu<1}
	From the proof of Theorem~\ref{Thm:local:bifurcation} and Lemma~\ref{Lem:phi'<0:phi<mu}, we have the bound $0 < \mu \le 1$.  Moreover, integrating equation~\eqref{eqv u eqn} on $\R$ and using the fact that $m(0)=1$, we obtain
	\[
	(\mu - 1) \int_\R \phi(x) \dx = \frac{1}{2} \int_\R \phi^2(x) \dx.
	\]
	Thus, the only way to reach $\mu=1$ is by approaching $\phi=0$, and hence when $s > 0$ is small, we have $0 < \mu < 1$.  By Proposition~\ref{Prop:Alternative(iii):not:occur}, the bifurcation curve does not intersect the trivial solution line for any $s > 0$, so $\mu(s) < 1$ for all $s$.
\end{remark}

In the next lemma, we show that along the bifurcation curve, the wave speed $\mu$ is, in fact, bounded away from $0$.  
\begin{lemma} \label{Lem:mu(s) gtrsim 1}
	Along the bifurcation curve in Theorem~\ref{Thm:Global:bifurcation} we have that
	\[
	\mu(s) \gtrsim 1
	\]
	uniformly for all $s \ge 0$.  
\end{lemma}
\begin{proof}
	By contradiction, assume that there exists a sequence $(s_n)_{k \in \N} \in \R_+$ with $\lim_{n \to \infty} s_n = \infty$ such that $\mu(s_n) \to 0$ as $n \to \infty$, while $\phi(s_n) \to \phi_0$ as $n \to \infty$ along the bifurcation curve $\mathfrak{G}$ found in Theorem~\ref{Thm:Global:bifurcation}.  In view of Lemma~\ref{Lem:bdd:closed:subsetS:compact}, there exists a subsequence of $(s_n)_{n \in \N}$ such that $\phi(s_n)$ converges to a solution $\phi_0$ of equation~\eqref{F defn}.  Along the bifurcation curve, we have that $\phi(s_n) < \mu(s_n) \to 0$ as $n \to \infty$, and it follows that $\phi_0 \le 0$.  In view of Lemma~\ref{Lem:maxPhi:minPhi:estimate}, we have $\max_x \phi_0(x) = 0$, whence $\phi_0 \equiv 0$ by Remark~\ref{Rmk:phi=0}.  Finally, Lemma~\ref{Lem:mu-phi:ge:lambdaP} leads to
	\[
	0 = \lim_{n \to \infty} (\mu(s_n) - \phi(s_n)(\pi)) \ge \lambda \pi > 0,
	\]
	which is a contradiction.  Thus, we have $\mu(s) \gtrsim 1$ uniformly for all $s \ge 0$.
\end{proof}
At this point, we know that the wave speed is bounded away from $0$ and $1$.  Before proving the main result, we show that alternatives $(i)$ and $(ii)$ in Theorem~\ref{Thm:Global:bifurcation} occur simultaneously.  
\begin{thm} \label{Thm:alter(i)(ii):occur}
	In Theorem \ref{Thm:Global:bifurcation}, alternatives $(i)$ and $(ii)$ both occur.  
\end{thm}
\begin{proof}
	Let $(\phi_k(s), \mu_k(s))$, $s \in \R$, be the bifurcation curve found in Theorem~\ref{Thm:Global:bifurcation}.  By Proposition~\ref{Prop:Alternative(iii):not:occur}, we know that $\phi_k(s)$ is nontrivial, even, and nonincreasing on $(0,\pi)$, and alternative $(iii)$ in Theorem~\ref{Thm:Global:bifurcation} does not occur.  Thus, it is either alternative $(i)$ or alternative $(ii)$ in Theorem~\ref{Thm:Global:bifurcation} occurs.  
	
	Suppose that alternative $(i)$ occurs.  Then we have either $\|\phi_k(s)\|_{C^\beta} \to \infty$ for some $\beta \in (1,2)$ or $|\mu_k(s)| \to \infty$ as $s \to \infty$.  By Remark~\ref{Rmk:0<mu<1}, $\mu_k(s)$ is bounded between $0$ and $1$, so the only possibility is $\|\phi_k(s)\|_{C^\beta(\T)}$ is unbounded.  In this case, alternative $(ii)$ must occur, that is, the quantity
	\[
	\liminf_{s \to \infty} \inf_{x \in \R} (\mu_k(s) - \phi_k(s)(x)) = 0
	\]
	holds.  Indeed, by contradiction, suppose that 
	\[
	\liminf_{s \to \infty} \inf_{x \in \R} (\mu_k(s) - \phi_k(s)(x)) \ge \delta
	\]
	for some $\delta > 0$.  For any such solution $(\phi_k, \mu_k)$, equation \eqref{main u eqn} gives
	\[
	|\phi_k(x) - \phi_k(y)| = \frac{2|L\phi_k(x) - L\phi_k(y)|}{2\mu_k - \phi_k(x) - \phi_k(y)} \le \frac{|L\phi_k(x) - L\phi_k(y)|}{\delta}.
	\]
	By Lemma \ref{Lem:Phi:estimate}, $\phi_k$ is uniformly bounded in $C(\T)$.  Then since $L$ maps $C^\beta(\T)$ to $C^{\beta + \alpha}$, there exists some constant $C > 0$ depending on $\delta$ and $L$ such that
	\[
	\|\phi_k\|_{C^\beta(\T)} \le C,
	\]
	which is a contradiction.  Therefore, alternative $(ii)$ must occur.
	
	On the other hand, suppose that alternative $(ii)$, but not alternative $(i)$, occurs.  Then fixing $k$, there exists a sequence $(\phi_k(s_n), \mu_k(s_n))_{n \in \N}$ in Theorem~\ref{Thm:Global:bifurcation} solving equation~\eqref{main u eqn} satisfying $\phi_k'(s_n) \le 0$ on $(0,\pi)$, $\phi_k(s_n) < \mu_k(s_n)$, and 
	\[
	\liminf_{n \to \infty} |\mu_k(s_n) - \phi_k(s_n)(0)| = 0,
	\]
	while $\phi_k(s_n)$ remains uniformly bounded in $C^\beta(\T)$ for $\beta \in (1,2)$.  Taking a limit along a subsequence in $C^{\beta'}$ for some $\beta' \in (1, \beta)$ yields a contradiction to Theorem~\ref{Thm:Soln:Phi:regularity:at:0}.  Therefore, both alternatives $(i)$ and $(ii)$ occur simultaneously.  
\end{proof}

We have all the ingredients to prove the main result in this paper.  The proof of Theorem~\ref{Thm:main} will show that the limiting wave at the end of the bifurcation curve is even, periodic, highest, and exactly Lipschitz at its crest.  
\begin{proof}[Proof of Theorem~\ref{Thm:main}]
	Let $(\phi_k(s), \mu_k(s))$ be the global bifurcation curve $\mathfrak{G}$ found in Theorem~\ref{Thm:Global:bifurcation} and let $(s_n)_{n \in \N}$ be a sequence in $\R_+$ tending to infinity.  By Lemma~\ref{Lem:mu(s) gtrsim 1} and Remark~\ref{Rmk:0<mu<1}, we know that $(\mu_k(s_n))_{n \in \N}$ is bounded between and away from $0$ and $1$.  Moreover, Lemma~\ref{Lem:global:bdd:sequence:has:convergent:subsequence} gives the existence of a subsequence $(\phi_k(s_{n_l}), \mu_k(s_{n_l}))$ converging uniformly to a solution $(\phi_0, \mu_0)$ as $l \to \infty$.  Finally, from Theorem~\ref{Thm:alter(i)(ii):occur} and Theorem~\ref{Thm:Soln:Phi:regularity:at:0}, we conclude that $\phi_0(0) = \mu_0$ with $\bar\phi$ being precisely Lipschitz continuous at the maximum point.  This finishes the proof of Theorem~\ref{Thm:main}.
\end{proof}

\section*{Acknowledgements}
H.L. would like to express his sincere gratitude to \thanks{Mats Ehrnstr\"om} for his continuous support and insightful comments.

\bibliography{bib.bib}
\bibliographystyle{plain}

\end{document}